\documentclass[preprint,3p,authoryear,times]{elsarticle}

\biboptions{authoryear}

\usepackage{amsmath}
\usepackage{amsfonts}
\usepackage{amssymb}
\usepackage{amsthm}
\usepackage{dsfont}
\usepackage{hyperref}

\usepackage{color}
\definecolor{britishracinggreen}{rgb}{0.0, 0.26, 0.15}
\definecolor{darkraspberry}{rgb}{0.53, 0.15, 0.34}

\newtheorem{thm}{Theorem}[section]
\newtheorem{lem}[thm]{Lemma}
\newtheorem{dfn}[thm]{Definition}
\newtheorem{prop}[thm]{Proposition}
\newtheorem{cor}[thm]{Corollary}

\theoremstyle{remark}

\newcommand{\var}{\operatorname{Var}}

\newcommand{\cov}{\operatorname{Cov}}

\providecommand{\Sn}{\mathcal{S}_n}
\providecommand{\expec}{\operatorname{\mathbb{E}}}
\newcommand{\prob}{\operatorname{\mathbb{P}}}

\newcommand{\eps}{\epsilon}
\newcommand{\1}{\mathds{1}}
\newcommand{\oh}{\operatorname{\mathrm{o}}}

\newcommand{\ceil}[1]{\lceil{#1}\rceil}
\newcommand{\argmax}{\operatornamewithlimits{\arg\max}}

\begin{document}

\begin{frontmatter}

\title{Maxima and near-maxima of a Gaussian random assignment field}

\author{Gilles Mordant\corref{cor1}}
\ead{gilles.mordant@uclouvain.be}
	
\author{Johan Segers}
\ead{johan.segers@uclouvain.be}
	
\cortext[cor1]{Correspnding author}
	
\address{LIDAM/ISBA, UCLouvain, Voie du Roman Pays 20, 1348 Louvain-la-Neuve, Belgium}

\begin{abstract}
	The assumption that the elements of the cost matrix in the classical assignment problem are drawn independently from a standard Gaussian distribution motivates the study of a particular Gaussian field indexed by the symmetric permutation group.
	The correlation structure of the field is determined by the Hamming distance between two permutations.
	The expectation of the maximum of the field is shown to go to infinity in the same way as if all variables of the field were independent.
	However, the variance of the maximum is shown to converge to zero at a rate which is slower than under independence, as the variance cannot be smaller than the one of the cost of the average assignment.
	Still, the convergence to zero of the variance means that the maximum possesses a property known as superconcentration.
	Finally, the dimension of the set of near-optimal assignments is shown to converge to zero.
\end{abstract}

\begin{keyword}
Extremal field \sep
Gaussian random field \sep
near maximal set \sep
random assignment \sep 
superconcentration 
\end{keyword}
\end{frontmatter}

\section{Introduction: the Gaussian random assignment field}

Optimal assignment is a classical problem appearing in mathematics and computer sciences that has also attracted probabilists' attention.
Consider the problem of assigning $n$ tasks to $n$ workers such that each task is assigned to one worker and each worker is given a single task.
The cost of matching task~$i$ to worker~$j$ is $c(i, j)$.
An assignment is represented by a permutation $u$ of $[n] = \{1, \ldots, n\}$, where $u(i) = j$ means that task~$i$ is scheduled for worker~$j$.
Let $\Sn$ be the symmetric group of permutations of $[n]$.
The question is which assignment $u \in \Sn$ has lowest total cost $\sum_{i \in [n]} c(i, u(i))$.
In the random assignment problem, the costs $c(i, j)$ are assumed to be drawn from some probability distribution.
For space considerations, we refer to \citet{aldous2004objective}, \citet{krokhmal2009random} or \citet[Section~2.4]{chatterjee2019general} for the rich history of the problem, some of the numerous applications and more comprehensive literature reviews. 

One of the most important results for the random assignment problem was obtained nearly two decades ago by \cite{Aldous}.
He showed that the expected minimal cost for independent unit exponential entries $c(i, j)$ converges to $\zeta(2)=\pi^2/6$ as $n \to \infty$, where $\zeta$ denotes the Riemann zeta function.
This result confirmed a longstanding conjecture by \citet{mezard1987solution} that could be numerically verified for large cost matrices by \citet{pardalos1993expected} thanks to refined interior point methods.
Aldous also proved an asymptotical essential uniqueness property, stating that every near-optimal assignment coincides with the optimal one except for a small proportion of pairs $(i, j)$.
Still for independent unit exponential costs, \cite{wastlund2005variance} showed that the variance of the minimal cost is asymptotically equivalent to $4 \left(  \zeta(2) - \zeta(3) \right) / n$ as $n \to \infty$.
For uniformly distributed costs, \cite{talagrand1995concentration} obtained an upper bound on the fluctuations of the minimal cost, whereas \cite{chatterjee2019general} obtained a lower bound for such fluctuations valid for a whole class of cost distributions on the positive half-line.
The two-way assignment problem we consider here is also qualified as linear, as opposed to multi-dimensional ones studied in \cite{krokhmal+g+p:2007}. In the quadratic version, for instance, the objective is to find an optimal assignment of $n$ jobs to $n$ workers on $n$ machines.

We consider the (linear) random assignment problem where the costs $c(i, j)$ are independent and identically distributed Gaussian random variables.
Up to the best of our knowledge, we are the first to do so.
For convenience, we standardize the problem in such a way that the cost of each assignment is standard Gaussian.
The collection of costs of all possible assignments is a Gaussian random vector indexed by $\Sn$.
We use the term Gaussian field because of the metric it induces on $\Sn$.
This motivates the following definition, which introduces the main object of this paper.

\begin{dfn}
	\label{dfn:graf}
	Let $n$ be a positive integer and let $\left(c(i, j) : i, j \in [n]\right)$ be a random matrix of independent standard Gaussian random variables. The \emph{Gaussian random assignment field} $g_n = \left(g_{n,u} : u \in \Sn\right)$ is defined by
	\[ 
		g_{n,u} = \frac{1}{\sqrt{n}} \sum_{i=1}^n c(i, u(i)). 
	\]
\end{dfn}

The normalization by $1/\sqrt{n}$ in the definition ensures that each variable $g_{n,u}$ is standard Gaussian.
Being a linear transformation of a vector of independent Gaussian random variables, $g_n$ has a multivariate Gaussian distribution, i.e., $g_n$ is a Gaussian random vector.
The correlation matrix $R_n = \left(r_n(u,v)\right)_{u,v \in \Sn}$ is easily seen to be
\begin{equation}
\label{eq:rnuv}
	r_n(u, v)
	= \expec(g_{n,u} g_{n,v}) 
	= \frac{1}{n} \left| \left\{ i \in [n] : u(i) = v(i) \right\} \right|, 
	\qquad u, v \in \Sn.
\end{equation}

The Hamming distance between two permutations $u, v \in \Sn$ is $d_H(u, v) = \left| \left\{ i \in [n] : u(i) \ne v(i) \right\} \right|$. The $L_2$ metric induced by $g_n$ on $\Sn$ is related to the Hamming distance via
\begin{equation}
\label{eq:metric}
	\sqrt{\expec\left( (g_{n,u} - g_{n,v})^2 \right)} 
	= \sqrt{2\left(1 - r_n(u, v)\right)}
	= \sqrt{2d_H(u, v)/n}.
\end{equation}
Two permutations $u$ and $v$ are far apart if the variables $g_{n,u}$ and $g_{n,v}$ are nearly orthogonal, which happens if $u(i)$ and $v(i)$ are different for most $i \in [n]$. This metric structure motivates the use of the term random `field'.

Already since Pierre R\'emond de Montmort's famous 1713 treatise %
\emph{Essay d'analyse sur les jeux de hazard}, it is known that, for any permutation $u \in \Sn$, the proportion of permutations $v \in \Sn$ such that the set $\{i \in [n] : u(i) = v(i)\}$ of fixed points with $u$ has size $k \in \{0,1,\ldots,n\}$ is equal to
\begin{equation}
\label{eq:deM}
	 \frac{1}{k!} \sum_{\ell=0}^{n-k} \frac{(-1)^\ell}{\ell!},
\end{equation}
a number which quickly converges to $e^{-1} / k!$ as $n \to \infty$. It follows that for large $n$, most assigments $g_{n,u}$ and $g_{n,v}$ are nearly uncorrelated---the probability that a unit-Poisson random variable is not larger than $4$ is higher than $99\%$.
Also, the average over all $(n!)^2$ correlations $r_n(u, v)$ is equal to $1/n$; see the proof of Proposition~\ref{thm: expec} below.

Still, the Gaussian random assignment field $g_n$ comprises $n!$ dependent Gaussian variables generated out of only $n^2$ independent ones.
There are thus two conflicting intuitions regarding the dependence within $g_n$.
On the one hand, it is a high-dimensional random vector supported by a relatively low-dimensional subspace.
On the other hand, two arbitrarily chosen components are nearly uncorrelated.
As $n$ grows, each of both effects becomes stronger.
The question is then which of the two tendencies dominates.
The answer may depend on the functional of $g_n$ of interest.

In the context of random assignment problems it is natural to study the lowest and highest achievable costs, i.e., the minimum and maximum over all assignments:
\begin{equation}
\label{eq:Mn}
	W_n = \min_{u \in \Sn} g_{n,u} \qquad \text{and} \qquad 
	M_n = \max_{u \in \Sn} g_{n,u}.
\end{equation}
By symmetry, $W_n$ has the same distribution as $-M_n$. We therefore only study the maximum $M_n$, as is common in the theory of Gaussian fields. Rather than as a cost, we could view $c(i, j)$ as the gain (when positive) or loss (when negative) of assigning `player' $i$ to `game' $j$. The highest possible gain over all assignments is then equal to $M_n$.

Our first objective is to study the asymptotic behavior of the expectation and the variance of $M_n$.
Theorem~\ref{thm:super}, our first main result, states that $\expec(M_n)$ is asymptotically equivalent to $\sqrt{2 \log(n!)}$. This rate is the same as the one for the maximum of $n!$ independent standard Gaussian variables, and is in fact the fastest rate possible for a centered Gaussian field with unit variances \citep[Eq.~(A.3)]{chatterjee2014superconcentration}.
In the terminology of \citet[Chapter~8]{chatterjee2014superconcentration}, this means that the Gaussian random assignment field $g_n$ is \emph{extremal}.
Moreover, by Corollary~\ref{cor:super}, $\var(M_n)$ converges to zero as $n \to \infty$.
Since each individual variable $g_{n,u}$ has unit variance, this means that, according to the same reference, $M_n$ is \emph{superconcentrated}. We refer to \cite{chatterjee2014superconcentration} for an in-depth treatment of the implications of the superconcentration phenomenon. According to Theorem~1.18 in that book, for instance, with high probability there are a large number of assignments $u \in \Sn$ that are mutually distant from each other according to the metric~\eqref{eq:metric} but which are near-maximal, i.e., the relative difference between $g_{n,u}$ and $M_n$ is small.
\cite{tanguy2015} illustrates the superconcentration phenomenon for extrema of stationary Gaussian processes and some other Gaussian models.

Still, the lower bound $\var(M_n) \ge 1/n$ that we derive in Proposition~\ref{thm: expec} implies that the rate at which $\var(M_n)$ converges to zero is slower than the one of the maximum of $n!$ independent standard Gaussian variables, which is $1/\log(n!) \sim 1/(n \log n)$ instead.
The lower bound is also sharper than the bound $\var(M_n) \ge c / \log(n!)$, with $c$ a universal positive constant, that follows from \citet[Corollary~1.9]{ding2015multiple}.
The tool to find our variance lower bound is elementary and works by `sourcing out' the average cost $\bar{g}_n = (n!)^{-1} \sum_{u \in \Sn} g_{n,u}$.

Our second topic of interest is the number of assignments which are nearly optimal. The following definition comes from \cite[Section~8.3]{chatterjee2014superconcentration}.

\begin{dfn}
	\label{dfn:nms}
	The \emph{near maximal set} of a Gaussian field $g = (g_1,\ldots,g_N)$ with $M = \max_{u \in [N]} g_u$ and $m = \expec(M)$ is
	\[
		A(\eps) = \left\{ u \in [N] : g_u > (1 - \eps) m \right\}, 
		\qquad \eps > 0.
	\]
\end{dfn}

The ratio $\log \left| A(\eps) \right| / \log N$ can be thought of as the (fractal) dimension of $A(\eps)$.
Let $A_n(\eps)$ be the near-maximal set of the Gaussian random assignment field $g_n$.
According to our second main result, Theorem~\ref{thm : main}, the expectation of $\log \left| A_n(\eps) \right|$, provided $A_n(\eps)$ is non-empty, is of smaller order than $\log(n!)$.
In particular, the expected dimension of $A_n(\eps)$ when non-empty converges to zero.

The main results are collected in Section~\ref{sec:main} and all proofs in Section~\ref{sec:proofs}.
A discussion in Section~\ref{sec:disc} of open questions concludes the paper.

\section{Main results}
\label{sec:main}

Some of our bounds on the expectation and the variance of the maximum $M_n$ of the Gaussian random assignment field $g_n$ rely on a simple principle. In a Gaussian random vector of which the covariance matrix has constant row sums, it is possible to write all variables as the sum of a common component and an independent Gaussian random vector. `Outsourcing' this common component produces bounds on the expectation and the variance of the maximum.

\begin{lem}
	\label{thm: Outsource}
	Let $g = (g_1,\ldots,g_N)$ be a centered Gaussian random field with $\var(g_u) = 1$ for all $u \in [N]$. Put $M = \max_{u\in[N]} g_u$. Suppose the row means $N^{-1} \sum_{v\in[N]} \cov(g_u, g_v)$ of the covariance matrix are the same for all $u \in [N]$ and let $s^2$ denote their common value. Then $0 \le s^2 \le 1$ and
	\begin{align*}
		\expec(M) &\le \sqrt{2(1-s^2)N},\\
		\var(M) &\ge s^2.
	\end{align*}
\end{lem}

\begin{prop}
	\label{thm: expec}
	For all integer $n \ge 1$, the maximum $M_n$ in \eqref{eq:Mn} of the Gaussian random assignment field $g_n$ in Definition~\ref{dfn:graf} satisfies
	\begin{align*}
		\expec(M_n) 
		&\le \sqrt{2 \left(1-1/n\right) \log (n!)}, \\
		\var(M_n) 
		&\geq 1/n.
	\end{align*}
\end{prop}

Since $g_n$ is a centered Gaussian field with unit variances, Eq.~(A.3) in \cite{chatterjee2014superconcentration} implies that $\expec(M_n) \le \sqrt{2 \log(n!)}$. Proposition~\ref{thm: expec} thus constitutes a modest improvement. In the introduction, we already noted that the lower bound for the variance improves the inequality $\var(M_n) \ge c / (n \log n)$ that follows from Corollary~1.9 in \cite{ding2015multiple}.

The question is at what rate $\expec(M_n)$ goes to infinity.
On the one hand, Proposition~\ref{thm: expec} implies that it cannot do so faster than $\sqrt{2 \log(n!)}$.
On the other hand, Theorem~8.3 in \cite{chatterjee2014superconcentration} yields the existence of a universal positive constant $C$ such that
\[
	\expec(M_n) \ge \sqrt{2 \log(n!)} - C \left(
	\log \log (n!) +
	\log \left(
		\sum_{u,v \in \Sn} (n!)^{-2/(1+r_n(u,v))}
	\right)
	\right)^{1/2}
\]
with $r_n(u, v)$ the pairwise correlation in Eq.~\eqref{eq:rnuv}.
A careful analysis of the last sum in the lower bound shows that $\liminf_{n \to \infty} \expec(M_n) / \sqrt{2 \log(n!) } > 0$.
Our first main result, Theorem~\ref{thm:super} states that the limit exists and is equal to one.
	
\begin{thm}
	\label{thm:super}
	Let $M_n$ be the maximum of the Gaussian random assignment field $g_n$. We have
	\begin{equation*}
		\lim_{n \to \infty} \frac{\expec(M_n)}{\sqrt{2 \log(n!)}} = 1.
	\end{equation*}
\end{thm}

Theorem~\ref{thm:super} in combination with \cite[Theorem~8.1]{chatterjee2014superconcentration} implies that $\var(M_n)$ tends to zero and thus that $g_n$ is \emph{superconcentrated}, as discussed in the introduction.

\begin{cor}
	\label{cor:super}
	For $M_n$ and $g_n$ as in Theorem~\ref{thm:super}, we have
	\[
		\lim_{n \to \infty} \var(M_n) = 0.
	\]
\end{cor}

Next we study the expected size of the near maximal set $A_n(\eps)$ in Definition~\ref{dfn:nms} of the Gaussian random assignment field $g_n$. Note that the total number of variables is $|\Sn| = n!$ and that $\log(n!)$ is asymptotically equivalent to $n \log(n)$ as $n \to \infty$. According to our second main result, Theorem~\ref{thm : main}, the expected size of $A_n(\eps)$ is substantially smaller. Since $A_n(\eps)$ can be empty with positive probability, we exclude this case from the expectation of its logarithm and we write $\expec[X; A] = \expec[X \1_A]$ for a random variable $X$ and an event $A$ with indicator $\1_A$.

\begin{thm}
	\label{thm : main}
	Let $g_n$ be the Gaussian random assignment field and let $A_n(\eps)$ be its near maximal set. There exist universal constants $C'$ and $C''$ such that
	\[
	\expec\left(\log \left\lvert A_n(\epsilon) \right\rvert ; A_n(\epsilon)\neq \varnothing \right) \le 
	\begin{cases}  
	C' \left(n\log n\right)^{3/4}
	& \text{if $0 < \epsilon \le \left(2 n \log n\right)^{-1/2}$,} \\  
	C'' \sqrt{\epsilon} \left(n\log n\right)
	& \text{if $\left(2 n \log n\right)^{-1/2} < \eps < 1$.}
	\end{cases}
	\]
\end{thm}

It follows from Theorem~\ref{thm : main} that the expected dimension $\log \left| A_n(\eps) \right| / \log(n!)$ of $A_n(\eps)$ when non-empty converges to zero for $\eps = \eps_n \to 0$.

\section{Proofs}
\label{sec:proofs}

\begin{lem}
	\label{lem: Outsource}
	Let $g = (g_1,\ldots,g_N)$ be a centred $N$-dimensional Gaussian random vector with covariance matrix $\Gamma = (\gamma_{uv})_{u,v\in[N]}$ such that the row sums $\sum_{v\in[N]} \gamma_{uv}$ do not depend on $u \in [N]$. Let $s^2 = N^{-1} \sum_{v\in[N]} \gamma_{uv}$ denote the common value of the row means. Then $0 \le s^2 \le \min_{u \in [N]} \gamma_{uu}$ and we have the decomposition
	\begin{equation} 
	\label{eq:gudecomp}
		g_u = \bar{g} + h_u, \qquad u \in [N],
	\end{equation}
	where $\bar{g} = N^{-1} \sum_{u \in N} g_u$ is a centred Gaussian random variable with variance $s^2$ and $h = (h_1,\ldots,h_N)$ is a centred Gaussian random vector with $\var(h_u) = \gamma_{uu} - s^2$, with $\bar{g}$ and $h$ being independent.
\end{lem}

\begin{proof}
	The covariance between each $g_u$ and the arithmetic mean $\bar{g}$ is
	\[
		\cov(g_u, \bar{g}) = \frac{1}{N} \sum_{v \in [N]} \cov(g_u, g_v) = s^2, \qquad u \in [N].
	\]
	This common covariance is also equal to the variance of $\bar{g}$, since 
	\[ 
		\var( \bar{g} ) = \frac{1}{N} \sum_{u \in [N]} \cov(g_u, \bar{g}) = s^2.
	\]
	As a consequence, $s^2$ is non-negative and the variables $\bar{g}$ and $g_u - \bar{g}$ are uncorrelated:
	\[
		\cov\left( \bar{g}, g_u - \bar{g} \right) = s^2 - s^2 = 0.
	\]
	It follows that for all $u \in [N]$ we have
	\[
		\gamma_{uu} = \var(g_u) = \var(\bar{g}) + \var(g_u - \bar{g}) = s^2 + \var(g_u - \bar{g})
	\]
	and thus
	\[
		\gamma_{uu} - s^2 = \var(g_u - \bar{g}) \ge 0.
	\]
	Define $h = (h_1,\ldots,h_N)$ by $h_u = g_u - \bar{g}$ for $u \in [N]$.
	Since the $(N+1)$-dimensional random vector $(\bar{g}, h)$ is jointly Gaussian, the fact that $\bar{g}$ is uncorrelated with every $h_u$ implies that $\bar{g}$ and $h$ are independent.
\end{proof}

\begin{proof}[Proof of Lemma~\ref{thm: Outsource}]
	We apply Lemma~\ref{lem: Outsource}. Since all variables $g_u$ have unit variance, it follows that $0 \le s^2 \le 1$. Let $\bar{g}$ and $h$ be as in Lemma~\ref{lem: Outsource}. In view of the decomposition~\eqref{eq:gudecomp}, the maximum $M$ over the Gaussian field $g$ is equal to
	\[
		M = \bar{g} + L \qquad \text{where} \qquad L = \max_{u \in [n]} h_u.
	\]
	The variables $\bar{g}$ and $L$ are independent, whence 
	\[ 
		\var(M) = \var(\bar{g}) + \var(L) \ge \var(\bar{g}) = s^2. 
	\]
	Furthermore, 
	\[ 
		\expec(M) = \expec(\bar{g}) + \expec(L) = \expec(L). 
	\]
	If $s^2 = 1$, then $h$ is degenerate at the origin and $L = 0$ almost surely. If $s^2 < 1$, then we can write $h = \sqrt{1 - s^2} \tilde{h}$ where $\tilde{h} = (\tilde{h}_1,\ldots,\tilde{h}_N)$ is a centred Gaussian field with $\var(\tilde{h}_u) = 1$ for all $u \in [N]$. We have $L = \sqrt{1 - s^2} \tilde{L}$ with $\tilde{L} = \max_{u \in [N]} \tilde{h}_u$. By Eq.~(A.3) in  \cite{chatterjee2014superconcentration}, we have $\expec(\tilde{L}) \le \sqrt{2 \log(N)}$. We conclude that 
	\[ 
		\expec(M) = \sqrt{1 - s^2} \expec(\tilde{L}) \le \sqrt{2(1-s^2) \log(N)}.
		\qedhere 
	\]
\end{proof}

\begin{proof}[Proof of Proposition~\ref{thm: expec}]
	We apply Lemma~\ref{thm: Outsource}.
	By permutation symmetry, the covariance matrix of $g_n$ has constant row sums: for $u, v, w \in \Sn$, we have $\cov(g_{u,n}, g_{v,n}) = \cov(g_{w \circ u, n}, g_{w \circ v, n})$, and thus, for $u_1, u_2 \in \Sn$,
	\[
		\sum_{v \in \Sn} \cov(g_{u_1,n}, g_{v,n})
		= \sum_{v \in \Sn} \cov(g_{u_2,n}, g_{u_2 \circ u_1^{-1} \circ v, n})
		= \sum_{v'\in \Sn} \cov(g_{u_2,n}, g_{v', n}).
	\]
	The common value $s^2$ of the row means of the covariance matrix of $g_n$ is equal to the variance of the average $\bar{g}_n = (n!)^{-1} \sum_{u \in \Sn} g_u$ over all assignments.
	But the latter is proportional to the average over all individual costs:
	\begin{align*}
		\bar{g}_n 
		&= \frac{1}{n!} \sum_{u \in \Sn} \frac{1}{\sqrt{n}} \sum_{i \in [n]} c(i, u(i)) \\
		&= \frac{1}{n! \sqrt{n}} \sum_{i \in [n]} \sum_{u \in \Sn} c(i, u(i)) \\
		&= \frac{1}{n! \sqrt{n}} \sum_{i \in [n]} \sum_{j \in [n]} c(i, j) 
		\underbrace{\left| \left\{ u \in \Sn : u(i) = j \right\} \right|}_{=(n-1)!} \\
		&= \frac{1}{n\sqrt{n}} \sum_{i \in [n]} \sum_{j \in [n]} c(i, j).
	\end{align*}
	It follows that the common value of the row means of the covariance matrix of $g_n$ is
	\[
		s^2 
		= \var(\bar{g}_n) 
		= \frac{1}{n^3} \cdot n^2 \cdot 1 
		= \frac{1}{n}.
	\]
	Lemma~\ref{thm: Outsource} yields the inequalities $\expec(M_n) \le \sqrt{2(1 - n^{-1})\log(n!)}$ and $\var(M_n) \ge n^{-1}$ for all positive integer $n$.
\end{proof}

\begin{proof}[Proof of Theorem~\ref{thm:super}]
	In view of Theorem~\ref{thm: expec}, it suffices to show that
	\begin{equation}
	\label{eq:EMn:liminf}
	\liminf_{n \to \infty} \frac{\expec(M_n)}{\sqrt{2 \log(n!)}} \ge 1.
	\end{equation}	
	To do so, we construct a lower bound to $M_n$ by a greedy approach.
	Define a random permutation $u_c \in \Sn$ recursively by passing through the rows $i$ of the cost matrix $c$ one after the other, each time discarding the columns $j$ that have already been selected in the previous steps:
	\begin{itemize}
		\item
		Row $i = 1$: let $u_c(1)$ be the index $j \in [n]$ that maximizes $c(1, j)$.
		\item
		Row $i = 2$: let $u_c(2)$ be the index $j \in [n] \setminus \{u_c(1)\}$ that maximizes $c(2, j)$.
		\item
		$\ldots$
		\item
		Row $i = n$: let $u_c(n)$ be the only remaining element in $[n] \setminus \{u_c(1),\ldots,u_c(n-1)\}$.
	\end{itemize}
	Formally, define $u_c \in \Sn$ by
	\begin{align*}
	u_c(1) &= \argmax_{j \in [n]} c(1, j), \\
	u_c(i) &= \argmax_{j \in [n] \setminus \{u_c(1), \ldots, u_c(i-1)\}} c(i, j),
	\qquad i \in \{2, \ldots, n\}.
	\end{align*}
	In case of ties, always choose the smallest index available; another tie-splitting rule would work too.
	Clearly,
	\begin{align*}
	c(1, u_c(1)) &= \max_{j \in [n]} c(1, j), \\
	c(i, u_c(i)) &= \max_{j \in [n] \setminus \{u_c(1),\ldots,u_c(i-1)\}} c(i, j),
	\qquad i \in \{2,\ldots,n\}.
	\end{align*}
	Then $c(1, u_c(1))$ is the maximum of $n$ independent standard Gaussian variables; $c(2, u_c(2))$ is the maximum of $(n-1)$ independent standard Gaussian variables; etc.
	Indeed, let $i \ge 2$. Since the indices $u_c(1),\ldots,u_c(i-1)$ are a function of the first $(i-1)$ rows of the cost matrix $c$, the conditional distribution of the maximum $c(i, u_c(i))$ over the remaining $n-i+1$ variables $c(i,j)$ with indices $j$ in $[n] \setminus \{u_c(1),\ldots,u_c(i-1)\}$ in the $i$th row given the earlier chosen indices $u_c(1),\ldots,u_c(i-1)$ is independent of the values of those indices.
	
	The random assignment $u_c$ constructed by this greedy algorithm yields a lower bound to the maximum:
	\[
	M_n \ge \frac{1}{\sqrt{n}} \sum_{i=1}^n c(i, u_c(i)).
	\]
	Let $\mu_n$ be the expectation of the maximum of $n$ independent standard Gaussian variables, i.e., $\mu_n = \expec[\max_{j \in [n]} c(1, j)]$.
	It follows that
	\[
	\expec(M_n) \ge \frac{1}{\sqrt{n}} \sum_{i \in [n]} \expec[c(i, u_c(i))]
	= \frac{1}{\sqrt{n}} \sum_{i \in [n]} \mu_{n-i+1}.
	\]
	From classical extreme value theory, it is known that\footnote{In fact, more accurate expansions are available. In \cite{resnick1987}, for instance, combine Example~2 on page~71 with Proposition~2.1(iii) on page~77 and Exercise~2.1.2 on page~84.}
	\begin{equation}
	\label{eq:mun}
	\lim_{n \to \infty} \frac{\mu_n}{\sqrt{2 \log n}} = 1.
	\end{equation}
	Let $\delta \in (0, 1)$. Since $\mu_n$ is non-decreasing in $n$, we find
	\begin{equation*}
	\expec(M_n) 
	\ge \frac{1}{\sqrt{n}} \sum_{i \in [n]} \mu_{i} 
	\ge \frac{1}{\sqrt{n}} \sum_{i \in [n], i > \delta n} \mu_{i} \\
	\ge \frac{1}{\sqrt{n}} \left(n - \ceil{\delta n} + 1\right) \mu_{\ceil{\delta n}},
	\end{equation*}
	where $\ceil{x}$ is the smallest integer not less than the scalar $x$. 
	By~\eqref{eq:mun}, we have
	\begin{align*}
	\frac{\mu_{\ceil{\delta n}}}{\sqrt{2 \log n}}
	= \frac{\mu_{\ceil{\delta n}}}{\sqrt{2 \log \ceil{\delta n}}}
	\sqrt{\frac{\log \ceil{\delta n}}{\log n}}
	\to 1, \qquad n \to \infty.
	\end{align*}
	We arrive at
	\[
	\frac{\expec(M_n)}{\sqrt{2 n \log n}}
	\ge 
	\frac{n - \ceil{\delta n} + 1}{n} 
	\frac{\mu_{\ceil{\delta n}}}{\sqrt{2 \log n}}
	\to 1-\delta, \qquad n \to \infty.
	\]
	As $\delta \in (0, 1)$ was arbitrary, we find
	\[
	\liminf_{n \to \infty} \frac{\expec(M_n)}{\sqrt{2 n \log n}} \ge 1.
	\]
	Since $\log(n!) / \left(n \log n\right) \to 1$ as $n \to \infty$, the claim~\eqref{eq:EMn:liminf} follows.
\end{proof}

\begin{proof}[Proof of Corollary~\ref{cor:super}]
	Combine Theorem~\ref{thm:super} with Theorem~8.1 in \cite{chatterjee2014superconcentration}, stating an upper bound for $\var(M_n)$ which tends to zero since $\expec(M_n) / \sqrt{2 \log(n!)}$ tends to one.
\end{proof}

As a preparation to the proof of Theorem~\ref{thm : main}, we state Theorem~12.4 in \cite{chatterjee2014superconcentration}.
The theorem gives a bound on the size of the near maximal set $A(\eps)$ in Definition~\ref{dfn:nms} in terms of two ingredients: 
the expectation $m$ of the maximum and the number $V(\delta)$ of variables that have a correlation larger than $1-\delta$ with a given variable.

\begin{thm}[Chatterjee]
	\label{thm:chatterjee}
	Let $g = (g_1,\ldots,g_N)$ be a centered Gaussian random vector such that $\var(g_u) = 1$ for all $u \in [N]$. Let $M$, $m$ and $A(\eps)$ be as in Definition~\ref{dfn:nms}. For $u \in [N]$ and $\delta \in (0, 1)$, define
	\begin{equation*} 
		B(u, \delta) = \left\{ v \in [N] : \cov(g_u, g_v) > 1 - \delta \right\}
		\qquad \text{and} \qquad
		V(\delta) = \max_{u \in [N]} \left| B(u, \delta) \right|
	\end{equation*}
	Then there exists a universal constant $C$ such that for any $\eps \in (0, 1)$ we have the bound
	\[
		\expec\left( \log \left| A(\eps) \right|; A(\eps) \ne 0 \right)
		\le
		\inf_{\delta \in (0, 1)} \left(
			\log V(\delta) + \frac{C \max \{ \eps m^2, m \}}{\delta}
		\right).
	\]
\end{thm}

\begin{proof}[Proof of Theorem~\ref{thm : main}]
	We apply Theorem~12.4 in \cite{chatterjee2014superconcentration}, stated for the reader's convenience as Theorem~\ref{thm:chatterjee} above, to the Gaussian random assignment field $g_n$.
	Recall the discussion around Eq.~\eqref{eq:deM}. For $\delta \in (0, 1)$, the size of the set 
	\[ 
		B_n(u, \delta) 
		= \left\{ v \in \Sn : \cov(g_u, g_v) \ge 1 - \delta \right\} 
		= \left\{ v \in \Sn : \left| \left\{ i \in [n] : u(i) = v(i) \right\} \right| > (1-\delta)n \right\}
	\] 
	does not depend on $u \in \Sn$ and is equal to
	\[
		V_n(\delta) = \left| B_n(u, \delta) \right|
		= n! \sum_{k \in [n], k > (1-\delta) n} \frac{1}{k!} \sum_{\ell = 0}^{n-k} \frac{(-1)^\ell}{\ell!}.
	\]
	Let $\ceil{x}$ denote the smallest integer not smaller than the scalar $x$. 
	Then $V_n(\delta)$ is bounded by
	\[
		V_n(\delta)
		\le \frac{n!}{\ceil{(1-\delta)n}!} \sum_{k \in [n]} \frac{1}{k!} \sum_{\ell = 0}^{n-k} \frac{(-1)^\ell}{\ell!}
		\le \frac{n!}{\ceil{(1-\delta)n}!}
		\le n^{\delta n}.
	\]
	Indeed, the third inequality follows from the fact that the number of integers $k$ such that $\ceil{(1-\delta)n} < k \le n$ is bounded by $n - (1-\delta)n = \delta n$, while the second inequality follows from
	\begin{align*}
		\sum_{k \in [n]} \frac{1}{k!} \sum_{\ell = 0}^{n-k} \frac{(-1)^\ell}{\ell!}
		&= \sum_{k \in [n]} \sum_{\ell = 0}^{n-k}
		\frac{1}{(k+\ell)!} \binom{k+\ell}{\ell} (-1)^\ell \\
		&= \sum_{s \in [n]} \frac{1}{s!} \sum_{\ell = 0}^{s-1} \binom{s}{\ell} (-1)^\ell \\
		&= \sum_{s \in [n]} \frac{1}{s!} \bigl( (1-1)^s - (-1)^s \bigr) 
		= \sum_{s \in [n]} \frac{1}{s!} (-1)^{s-1} 
		\le 1,
	\end{align*}
	in view of the binomial theorem.
	By Theorem~\ref{thm:chatterjee} and the bound on $V_n(\delta)$, it follows that
	\[
		\expec\left( \log A_n(\eps); A_n(\eps) \ne \varnothing \right)
		\le
		\inf_{\delta \in (0, 1)} \left(
			\delta n \log(n) + \frac{C \max \left\{\eps m_n^2, m_n\right\}}{\delta}
		\right)
	\]
	where $m_n = \expec(M_n)$. Since $m_n \le \sqrt{2 \log(n!)} \le \sqrt{2 n \log n}$ by Theorem~\ref{thm: expec}, we obtain
	\[
		\expec\left( \log A_n(\eps); A_n(\eps) \ne \varnothing \right)
		\le
		\inf_{\delta \in (0, 1)} \left(
			\delta n \log(n) + 
			\frac{C \max\left\{ 2 \eps n \log n, \sqrt{2 n \log n}\right\}}{\delta}
		\right).
	\]
	It remains to distinguish between the following two cases:
	\begin{enumerate}[(i)]
	\item 
		If $0 < \eps \le 1 / \sqrt{2 n \log n}$, then put $\delta = \left(n \log n\right)^{-1/4}$. Both terms in the sum are of the order $\left(n \log n\right)^{3/4}$.
	\item
		If $1 / \sqrt{2 n \log n} < \eps < 1$, then put $\delta = \sqrt{\eps}$. Both terms in the sum are of the order $\sqrt{\eps} n \log n$.
	\end{enumerate}
	The claim follows.
\end{proof}

\section{Discussion}
\label{sec:disc}

Proposition~\ref{thm: expec} in combination with Corollary~\ref{cor:super} yields $1/n \le \var(M_n) \to 0$ as $n \to \infty$.
The correct rate at which $\var(M_n)$ converges to zero remains an open problem.
For unit exponential costs, \cite{wastlund2005variance} showed that the variance of the minimal costs is asymptotically equivalent to a constant times $1/n$ as $n \to \infty$.
It is however unclear if this also occurs for the Gaussian random assignment field.

In Lemma~\ref{thm: Outsource}, we have shown that the variance of the maximum of a Gaussian array could be bounded from below by the variance of the sample mean under the fixed-row-sum assumption on the covariance matrix. What are (easily verifiable) necessary and sufficient conditions on the covariance matrix for this lower bound to provide the correct order of the variance of the maximum?

Finally, limit theorems for the maxima of Gaussian random variables are well-covered in the literature---see for instance \cite{Leadbetter} and the references therein---in particular in the weakly dependent case and for stationary sequences. 
In the latter case, assuming some regularity on the correlation function, \citet{mittal1975limit} proved that Gumbel limits, mixtures of Gumbel and Gaussian distributions and Gaussian distributions could arise depending on the speed at which correlations tend to zero. 
The maximum of the Gaussian random assignment field possesses characteristics both of a maximum (over all possible assignments) and of a sum (over all jobs). How does the limit distribution look like? Although the setting is different, the answer could give more insight into the limit distribution of the empirical Wasserstein distance \citep{panaretos+z:2019}, a much coveted result, which corresponds to a matching problem with a matrix of random but dependent costs.

\section*{Acknowledgments}

The first author thanks Sourav Chatterjee for some insightful remarks. The comments by Stefka Kirilova and Vincent Plassier are also gratefully acknowledged. The authors would also like to thank the Associate Editor and an anonymous Referee for comments that helped improve the presentation of the paper and for pointing out particularly relevant references. The second author gratefully acknowledges funding by FNRS-F.R.S.\ grant CDR~J.0146.19.

\bibliography{Near_Maximal_sets.bib}

\end{document}